\documentclass[12pt,a4paper]{article}
\usepackage{amssymb,amsfonts,amsthm,amsmath}
\usepackage{textcomp}
\usepackage{setspace}

\usepackage{graphicx}

\def\hang{\hangindent\parindent}
 \def\rf{\par\noindent\hang}

\newtheorem{theorem}{Theorem}

\newcommand{\bX}{\boldsymbol{X}}

\newcommand{\btheta}{\boldsymbol{\theta}}

\newcommand{\ntheta}{||\btheta||}

\newcommand{\aplsT}{a^{+}(T)}

\begin{document}

\baselineskip=18pt


\begin{center} \large{{\bf A new recentered confidence sphere for the multivariate normal mean}}
\end{center}

\bigskip



\begin{center}
\large \sc {Waruni Abeysekera and Paul Kabaila$^{*}$}
\end{center}

\begin{center}
{\it La Trobe University}
\end{center}

\vspace{12cm}


\noindent $^*$ Author to whom correspondence should be addressed.

\noindent Department of Mathematics and Statistics, La Trobe University, Victoria 3086, Australia.

\noindent e-mail: P.Kabaila@latrobe.edu.au

\noindent Facsimile: 3 9479 2466

\noindent Telephone: 3 9479 2594

\newpage

\begin{center}
\large{{\bf Abstract}}
\end{center}

\noindent  We describe a new recentered confidence sphere for the mean, $\btheta$,
of a multivariate normal distribution. This sphere is centred on the positive-part
James-Stein estimator, with radius that is a piecewise cubic Hermite interpolating
polynomial function of the norm of the data vector. This radius function is determined
by numerically minimizing the scaled expected volume, at $\btheta = \boldsymbol{0}$,
of this confidence sphere, subject to the coverage constraint. We use the computationally-convenient
formula, derived by Casella and Hwang, 1983, {\sl JASA}, for the coverage probability
of a recentered confidence sphere. Casella and Hwang, \textit{op. cit.}, describe a recentered confidence
sphere that is also centred on the positive-part
James-Stein estimator, but with radius function determined by empirical Bayes
considerations. Our new recentered confidence sphere compares favourably with this
confidence sphere, in terms of both the minimum coverage probability and the scaled expected
volume at $\btheta = \boldsymbol{0}$.

\bigskip


\newpage

\noindent {\large{\bf 1. Introduction}}

\medskip

Suppose that $\bX = \left( X_1,...,X_p \right) \sim N\left( \btheta, \sigma^2 \boldsymbol{\textbf{I}_{p}} \right)$
where
$\btheta = \left( \theta_1,...,\theta_p \right)$ and
$\boldsymbol{\textbf{I}_{p}}$ denotes the $p \times p$ identity matrix.
We consider the problem of finding a confidence set for $\btheta$, based on $\bX$, with prescribed
confidence coefficient $1-\alpha$ that satisfies the following two conditions. The first
condition is that this confidence set has volume that never exceeds the volume of the
standard $1-\alpha$ confidence sphere centred at $\bX$.
The second condition is that this confidence set has an expected volume that
is small by comparison with that
of the standard $1-\alpha$ confidence sphere, when $\btheta = \boldsymbol{0}$.
Casella and Hwang (1987) argue cogently that the confidence set for $\btheta$ should be
tailored to the uncertain prior information available about $\btheta$.
We consider confidence sets that are tailored to the uncertain prior information
that $\btheta = \boldsymbol{0}$. The second condition is motivated by a desire to
utilize this uncertain prior information.
This is a very difficult problem to solve and, in common with much of the existing literature on it,
we assume that $\sigma^2$ is known. Without loss of generality, we assume that $\sigma^2 = 1$.
The standard $1-\alpha$ confidence set for $\btheta$ is
$I = \{\btheta : \, \| \btheta-\bX \| \leq d \}$,
where the positive number $d$ satisfies
$P\big( Q \leq d^2 \big) = 1-\alpha$
for $Q \sim \chi_p^2$.

The early work of Stein (1962), Brown (1966), Joshi (1967) and Hwang and Casella (1982) that relates
to this problem is reviewed by Casella and Hwang (2012). Confidence sets for $\btheta$, of various shapes,
that solve this problem have been proposed by Faith (1976), Berger (1980),
Casella and Hwang (1983, 1987), Shinozaki (1989), Tseng and Brown (1997), Samworth (2005)
and Efron (2006).
These confidence sets are reviewed and compared by Efron (2006) and Casella and Hwang (2012).

To provide specific proposals for recentered confidence spheres, Casella and Hwang (1983, 1987) and
Samworth (2005) center their confidence spheres at the positive-part James-Stein estimator.
The radii of these confidence spheres are functions of $|| \bX ||$ that are determined by
empirical Bayes considerations (Casella and Hwang, 1983), Taylor series or the bootstrap
(Samworth, 2005).
These papers compare confidence sets for $\btheta$
using the coverage probability and the scaled volume (or its $p$'th root).

In the present paper, we also consider a confidence sphere for $\btheta$ that is centered
at the positive-part James-Stein estimator.
However, we have chosen the radius of this confidence sphere to be a
piecewise cubic Hermite interpolating
polynomial function of $|| \bX ||$. This radius function is required to be both nondecreasing and
bounded above by $d$. This function is determined by numerically minimizing the
scaled expected volume of this confidence sphere at $\btheta = \boldsymbol{0}$, subject to the constraint that the
coverage probability never falls below $1-\alpha$. Here, the scaled expected volume is defined
to be the ratio (expected volume of the recentered confidence sphere) / (volume of $I$).
This numerical constrained minimization is made feasible by the fact that both the coverage
probability and the scaled expected volume of a recentered confidence sphere are
functions of $\gamma = || \btheta ||$, for a given radius function.
We use the computationally-convenient
formula, derived by Casella and Hwang (1983), for the coverage probability
of a recentered confidence sphere for $p$ odd.

The two main contributions of the present paper are the following:
\begin{enumerate}

\item[(1)]

We shift the focus from scaled volume (or its $p$'th root) to scaled {\sl expected} volume.
A goal of seeking to minimize
(in some sense) the volume of the recentered confidence sphere for the most probable values of
$\bX$ when $\btheta = \boldsymbol{0}$, subject to the coverage constraint, is problematic (Casella and Hwang, 1986).
By contrast, as we show in the present paper, minimization of the scaled expected volume of
the recentered confidence sphere (RCS) at $\btheta = \boldsymbol{0}$, subject to the coverage constraint,
leads to a RCS with excellent performance.

\item[(2)]

Our new RCS compares favourably with the RCS described in Section 4 of
Casella and Hwang (1983), in terms of both the minimum coverage probability
and the scaled expected volume at $\btheta = \boldsymbol{0}$.
This is shown in Figure 1, 2 and 3, which are for $1-\alpha=0.95$ and $p = 3, 5$ and $15$,
respectively.
The bottom panel and middle panel in each of these figures
compare the coverage probability and scaled expected volume, respectively,
of the new RCS and the RCS of Casella and Hwang (1983, Section 4).

\end{enumerate}

\baselineskip=17pt

\bigskip

\noindent {\large{\bf 2. Performance of the new RCS
by comparison with the RCS of Casella and Hwang (1983, Section 4)}}

\medskip

Both the RCS of Casella and Hwang (1983, Section 4) and the new RCS are centered on the positive-part
James-Stein estimator. These RCS's take the form
$J(b) = \big \{ \btheta : \| \aplsT \bX - \btheta \| \leq b(T) \big \}$, where
$T = ||\bX||/\sqrt{p}$,
\begin{equation*}
a^{+}(x)=\text{max}\left\{ 0, 1-\left(1-\frac{2}{p}\right)\frac{1}{x^2} \right\},
\end{equation*}
and $b: [0, \infty) \rightarrow (0, \infty)$ is required to be a nondecreasing function
that is bounded above by $d$. We call $b$ the radius function.
We use a slightly different notation for a RCS
from that used by Casella and Hwang (1983),
who express a RCS in
terms of $||\bX||$.

We define the scaled expected volume of $J(b)$ to be the ratio
\begin{equation}
\label{sev_initial}
\frac{E_{\btheta} (\text{volume of } J(b))}{\text{volume of } I}
= E_{\btheta} \left( \left(\frac{b(T)}{d} \right)^p \right),
\end{equation}
since the volume of a sphere in $\mathbb{R}^p$ with radius $r$ is $2 \, r^p \, \pi^{p/2}/\big(p \, \Gamma(p/2)\big)$.
As proved in the Appendix, this is a function of $\gamma = ||\btheta||$, for given function $b$.
It can also be shown that the coverage probability of $J(b)$
is a function of $\gamma$, for given function $b$.

For computational feasibility, we specify the following parametric form for the
radius function $b$. Suppose $b(x)=d$ for all $x \geq k$,
where $k$ is a (sufficiently large) specified positive number.
Suppose that $x_1,\dots,x_q$ satisfy $0 = x_1 < x_2 < \dots < x_q = k$. The function $b$ is fully
specified by the vector $b(x_1),\dots,b(x_q)$ as follows. The value of $b(x)$ for any given
$x \in [0, k]$ is found by piecewise cubic Hermite polynomial interpolation
for these given function values. We call $x_1, \dots,x_q$ the knots.

For judiciously-chosen values of the knots, we compute the function $b$, which takes
this parametric form, is nondecreasing and is bounded above by $d$,
such that (a) the scaled expected
volume evaluated at $\btheta = \boldsymbol{0}$ (i.e. at $\gamma= 0$) is minimized
and (b) the coverage
probability of $J(b)$ never falls below $1-\alpha$.

This coverage constraint is implemented in the
computations as follows. For any reasonable choice of the function $b$,
the coverage probability of $J(b)$ converges to $1-\alpha$ as $\gamma \rightarrow \infty$. The
constraints implemented in the computations are that the coverage probability of $J(b)$ is
greater than or equal to $1-\alpha$ for every $\gamma$ in a judiciously-chosen finite set of
values. That a given finite set of values of $\gamma$ is adequate to the task is judged by checking
numerically,
at the completion of computations, that the coverage probability constraint is satisfied
for all $\gamma \ge 0$.

This type of computation has been carried out in other related problems by
Farchione and Kabaila (2008, 2012) and Kabaila and Giri (2009ab, 2013ab). The main
lesson from these related computations is that the coverage probability needs to
be computed with great accuracy. Fortunately, in the context of the present paper,
the coverage probability of $J(b)$ can be computed with great accuracy for
$p$ odd, using the formula derived by Casella and Hwang (1983). Note that there
is a typographical error in this formula as stated on page 691 of Casella and Hwang (1983).
The $(n+1)!$ on the second line of (3.10) should be replaced by $(n+i)!$.
To help ensure accurate computation of the coverage probability, progressive
numerical quadrature, using Simpson's rule and a doubling of equal-length segments
at each stage of the progression is used.
The main stopping criterion is that $|Q_{2s} - Q_s|/Q_{2s} \le 10^{-8}$,
where $Q_s$ denotes the computed quadrature using $s$ segments.
All of the computations presented in the present paper were performed using
programs written in MATLAB using the Statistics and
Optimization toolboxes.

We now compare the coverage probability and scaled expected volume properties
of the new RCS and the RCS of Casella and Hwang (1983, Section 4), which
uses the following radius function
\begin{equation*}
b(x) =
\begin{cases}
\sqrt{( 1-(p-2)/d^2 )( d^2 - p \, \log ( 1-(p-2)/d^2 ))} &\text{for } x \in [0, d/\sqrt{p}] \\
\sqrt{( 1-(p-2)/(p \, x^2) ) ( d^2 - p \, \log ( 1-(p-2)/(p \, x^2) ) )} &\text{for } x > d/\sqrt{p}.
\end{cases}
\end{equation*}
This radius function was determined by empirical Bayes considerations.

\baselineskip=18pt

For $1-\alpha = 0.95$, the new RCS was computed for each $p \in \{3,5,7,9,11,13,15,17,19\}$.
We restrict to odd values of $p$ in order to use the computationally-convenient formula
for the coverage probability derived by Casella and Hwang (1983) for $p$ odd.
We have chosen $k=10$.
The positions of the knots of $b$ have been chosen to depend on $p$ as follows.
The first two knots are at 0 and $d/\sqrt{p}$. The next two knots are at
$(d/\sqrt{p}) + c$ and $(d/\sqrt{p}) + 2c$, where $c = ((k/2) - (d/\sqrt{p}))/3$.
The remaining knots are at $k/2$, $3k/4$ and $k$.
The coverage constraint was implemented in the computations by requiring that the
coverage probability of $J(b)$ is greater than or equal to $1-\alpha$ for all
$\gamma \in \{0,1,2,\dots,64,65\}$. This was shown to be adequate to the task by checking
numerically, at the completion of the computation of the new RCS,
that the coverage probability constraint
is satisfied for all $\gamma \ge 0$.

Figures 1, 2 and 3 compare the radius functions $b$, coverage probabilities
and scaled expected volumes of the new RCS and the RCS of
Casella and Hwang (1983, Section 4) for $1-\alpha = 0.95$ and $p = 3, 5$ and $15$,
respectively. The top panel of each of these figures compares the radius functions $b$.
The difference between the radius functions of these RCS's is substantial for $p=3$
and increases as $p$ increases. The middle panel
compares the coverage probabilities of these RCS's as functions of
$\gamma = || \btheta ||$. The bottom panel compares the scaled expected volumes
of these RCS's as functions of $\gamma$.

Figure 1 shows that, for $p=3$, the coverage probability (CP) of the new RCS is no less than $0.95$ for all
$\gamma$ values, while the CP of the RCS of Casella and Hwang (1983, Section 4)
is slightly below $0.95$ at some $\gamma$ values. This figure also shows that, for $p=3$, the
scaled expected volume (SEV) at $\btheta = \boldsymbol{0}$ of the new RCS is substantially
less than the SEV at $\btheta = \boldsymbol{0}$ of the RCS of Casella and Hwang (1983, Section 4).
Figure 2 tells a similar story for $p=5$. Figure 3, which is for the case
$p=15$, shows that the new RCS performs substantially
better than the RCS of Casella and Hwang (1983, Section 4), in terms of SEV at $\btheta = \boldsymbol{0}$.
However, for this value of $p$, both RCS's maintain a CP that is greater than or equal to $0.95$.
The computations of the CP and SEV of these RCS's were checked using Monte Carlo simulations.

Table 1 presents a comparison of the minimum CP's and the SEV's at $\btheta = \boldsymbol{0}$
of these RCS's for $p \in \{3,5,7,9,11,13,15,17,19\}$.
According to this table,
the new RCS always achieves a CP greater than or equal to $0.95$,
while the RCS of Casella and Hwang (1983, Section 4) does not achieve this for $p=3$ and $p=5$.
Also, for every value of $p$ considered, the new RCS achieves a substantially lower SEV at $\btheta = \boldsymbol{0}$
than the RCS of Casella and Hwang (1983, Section 4). In summary, our new RCS compares favourably with that of
Casella and Hwang (1983, Section 4), in terms of both the minimum CP and the SEV at $\btheta = \boldsymbol{0}$.

\newpage
\begin{figure}[h!]
\centering
\begin{tabular}{c}
\includegraphics[width=1.0\textwidth]{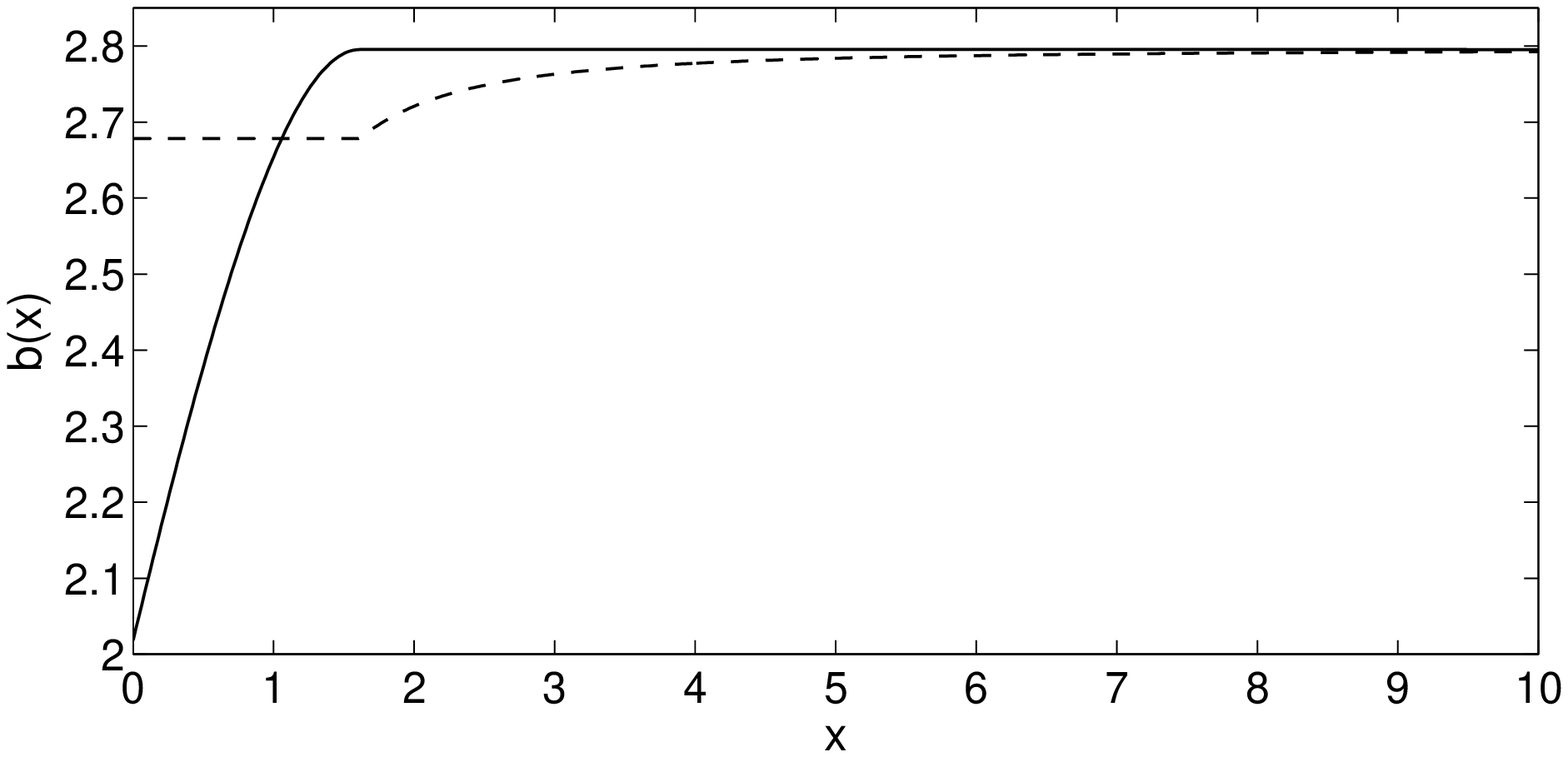} \\
\includegraphics[width=1.0\textwidth]{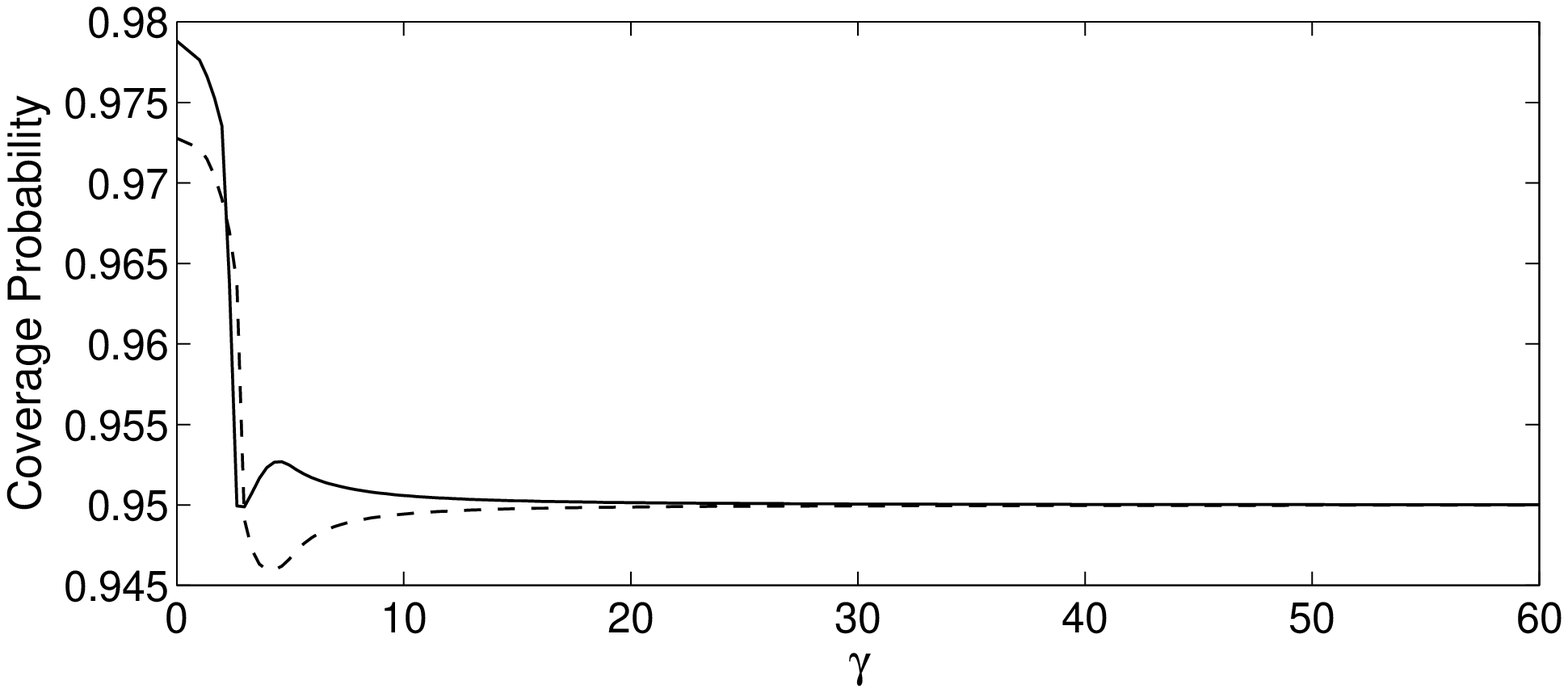}  \\
\includegraphics[width=1.0\textwidth]{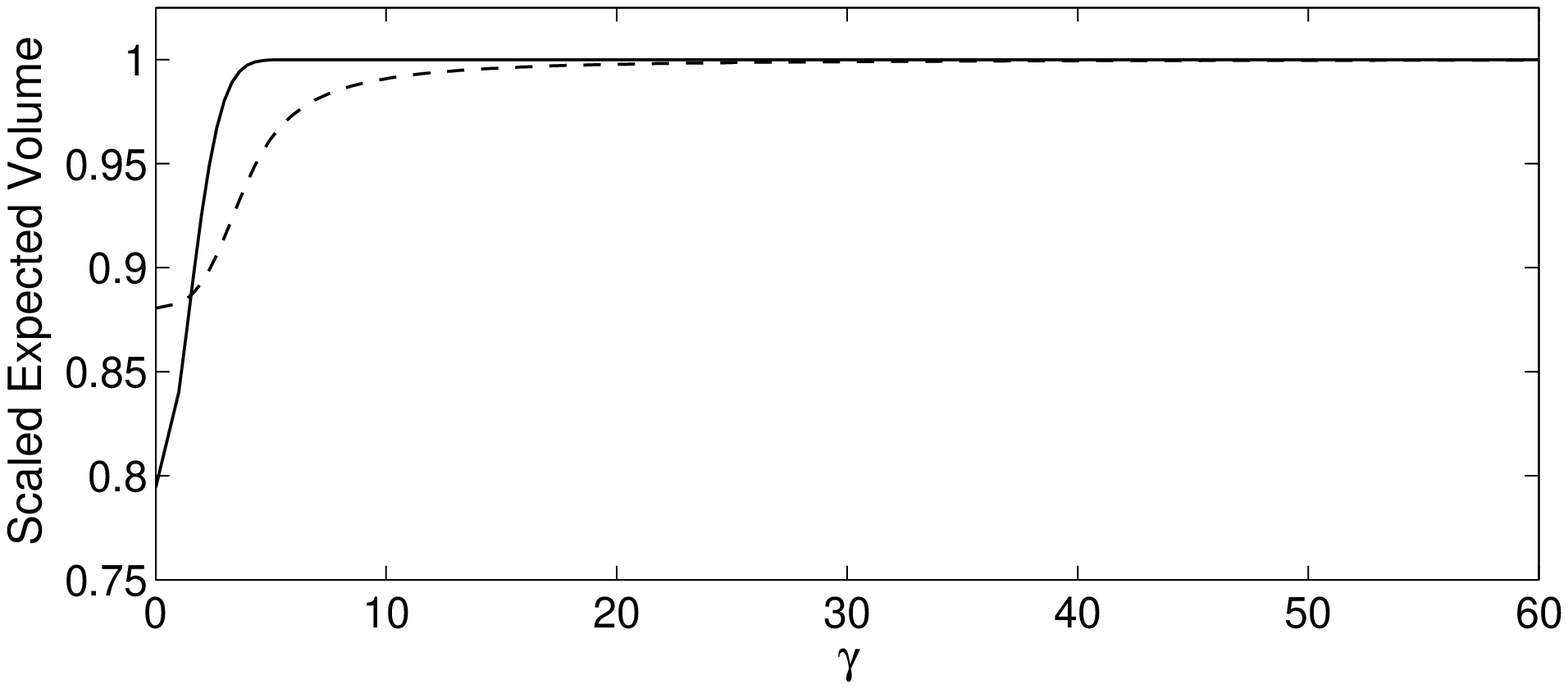} \\
Legend: ------ new RCS  \ \ \ - - -  RCS of Casella and Hwang (1983, Section 4)
\end{tabular}
\caption{Plots of the radius function $b$, coverage probability and
scaled expected volume (as functions of $\gamma = || \btheta ||$) for both the new RCS and
the RCS of Casella and Hwang (1983, Section 4),
for $1-\alpha=0.95$ and $p=3$.}
\label{figure 1}
\end{figure}
\clearpage
\newpage
\begin{figure}[h!]
\centering
\begin{tabular}{c}
\includegraphics[width=1.0\textwidth]{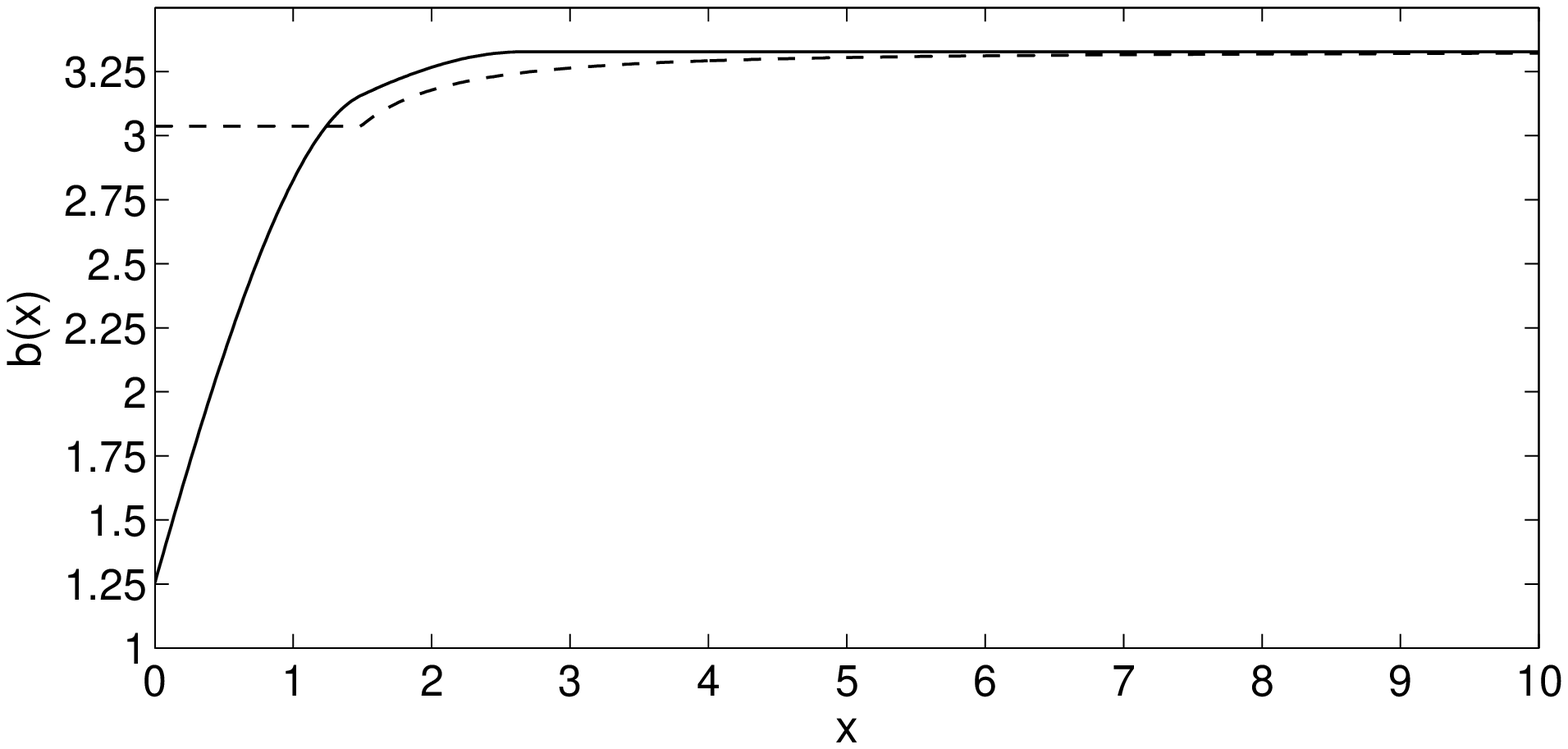} \\
\includegraphics[width=1.0\textwidth]{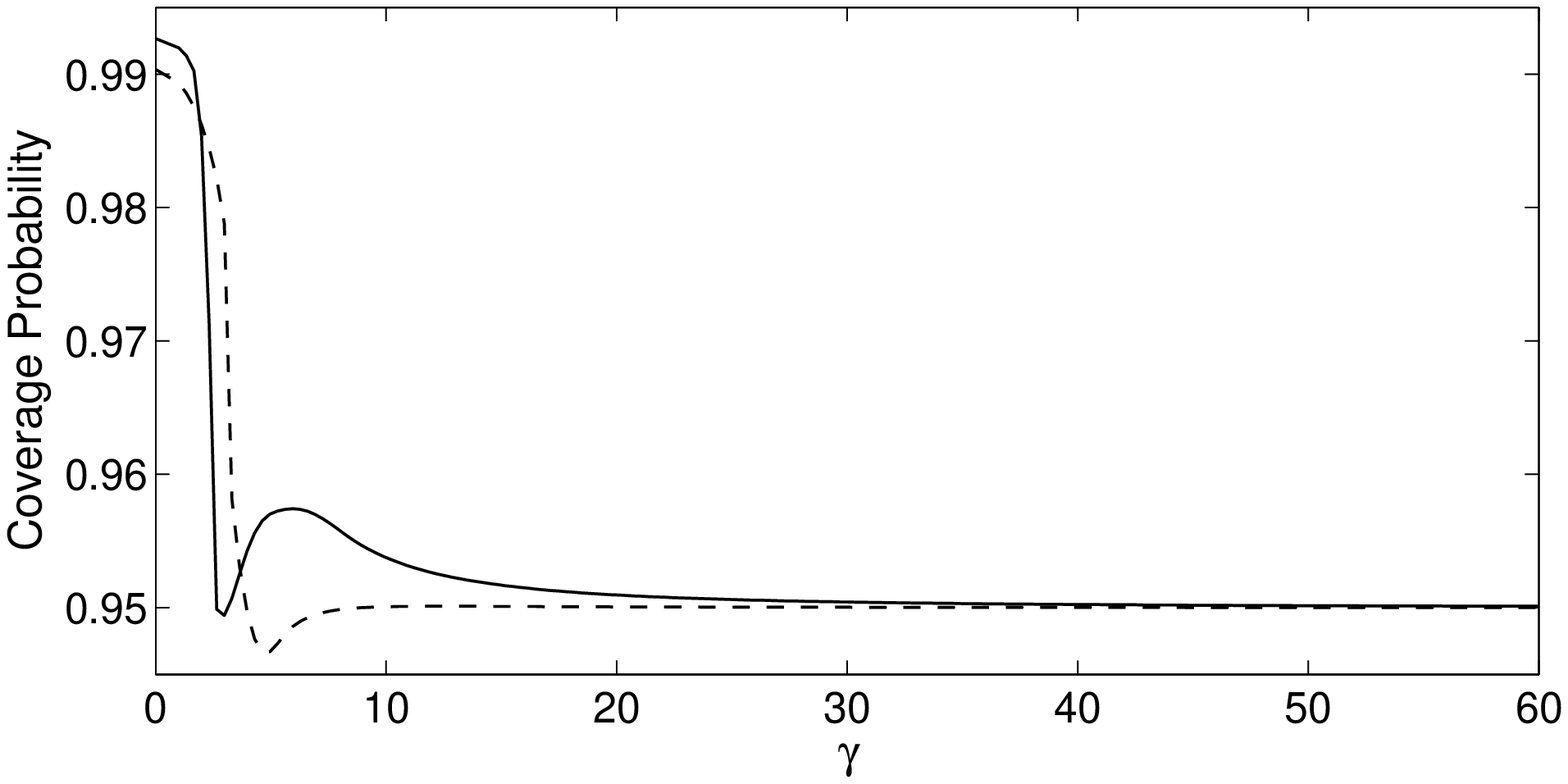}  \\
\includegraphics[width=1.0\textwidth]{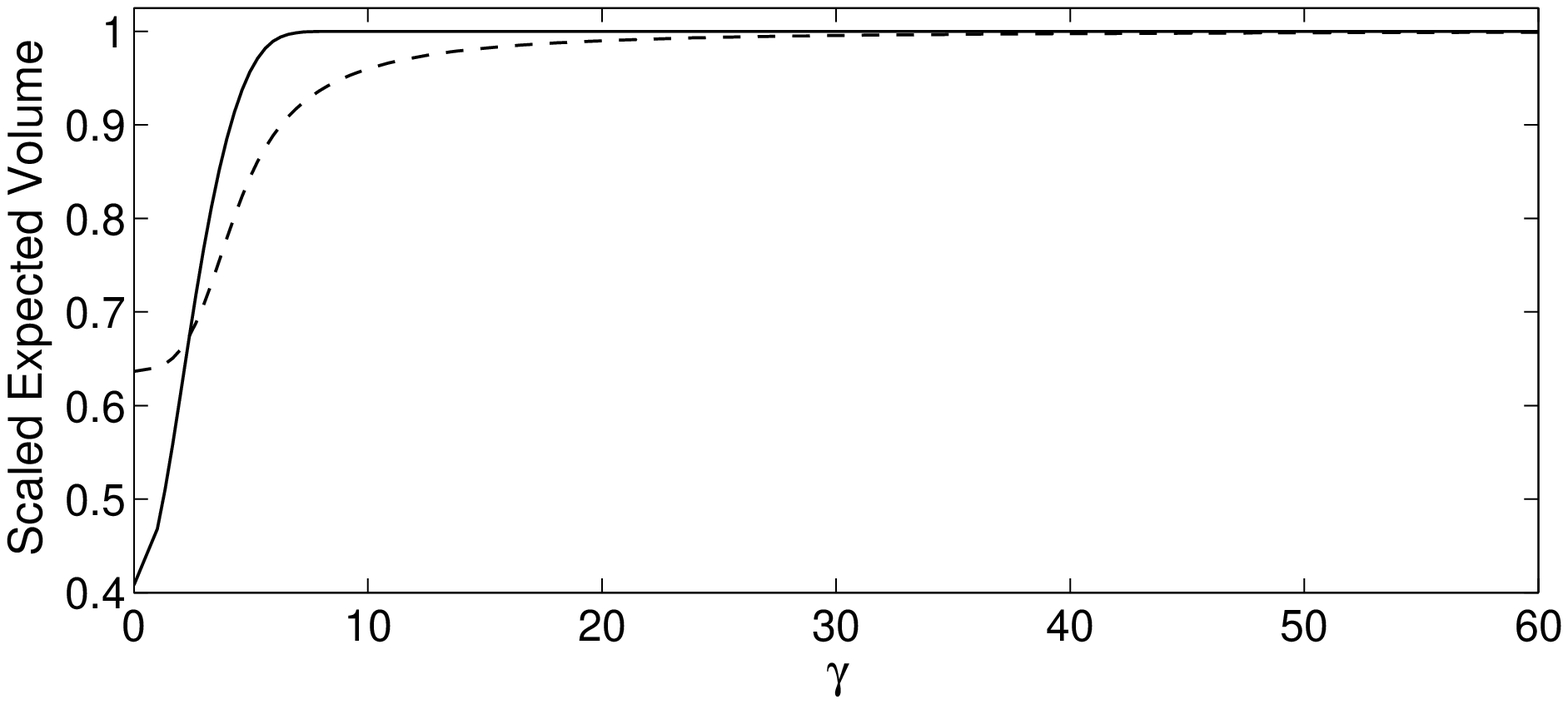} \\
Legend: ------ new RCS  \ \ \  - - -  RCS of Casella and Hwang (1983, Section 4)
\end{tabular}
\caption{Plots of the radius function $b$, coverage probability and
scaled expected volume (as functions of $\gamma = || \btheta ||$) for both the new RCS and
the RCS of Casella and Hwang (1983, Section 4),
for $1-\alpha=0.95$ and $p=5$.}
\label{figure 2}
\end{figure}
\clearpage
\newpage
\begin{figure}[h!]
\centering
\begin{tabular}{c}
\includegraphics[width=1.0\textwidth]{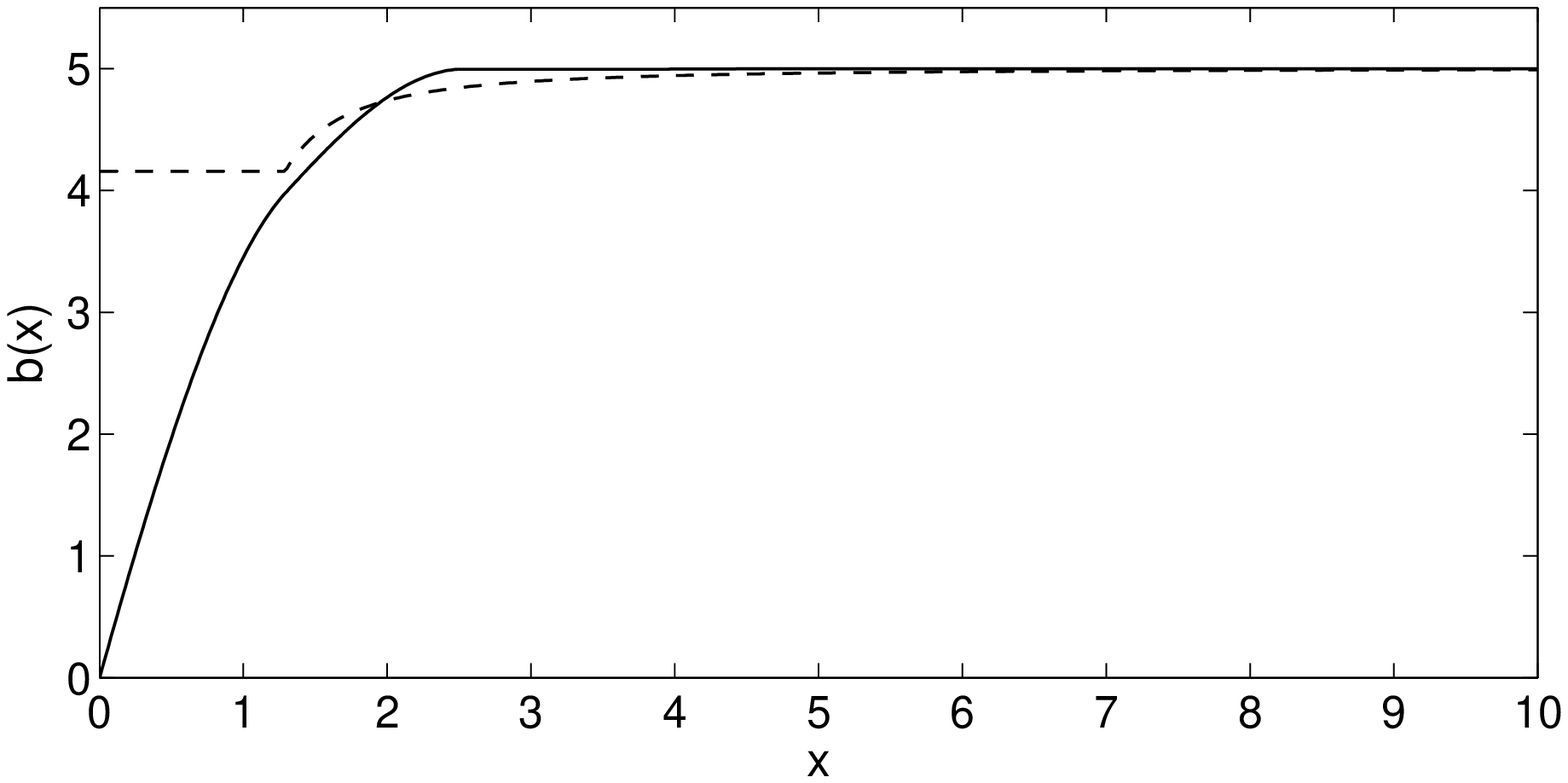} \\
\includegraphics[width=1.0\textwidth]{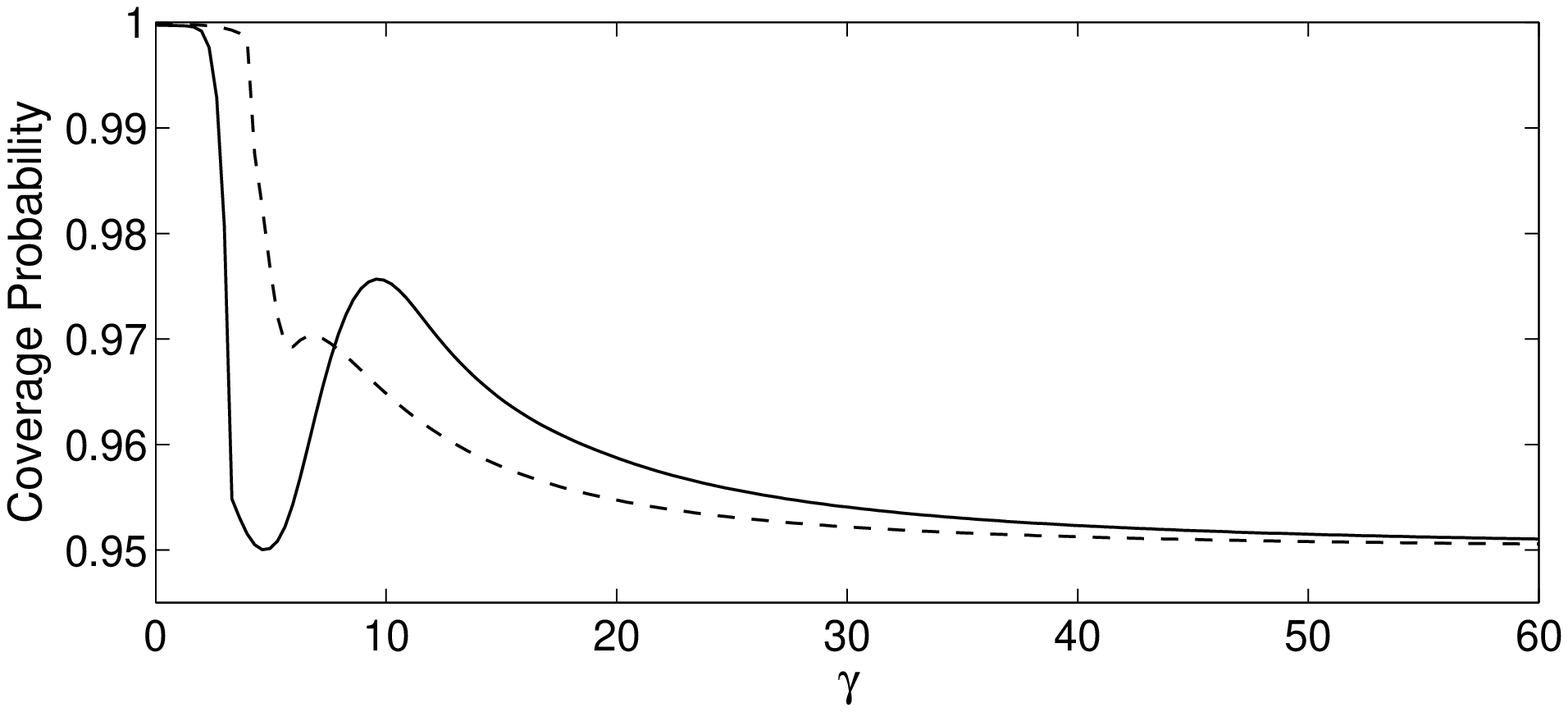}  \\
\includegraphics[width=1.0\textwidth]{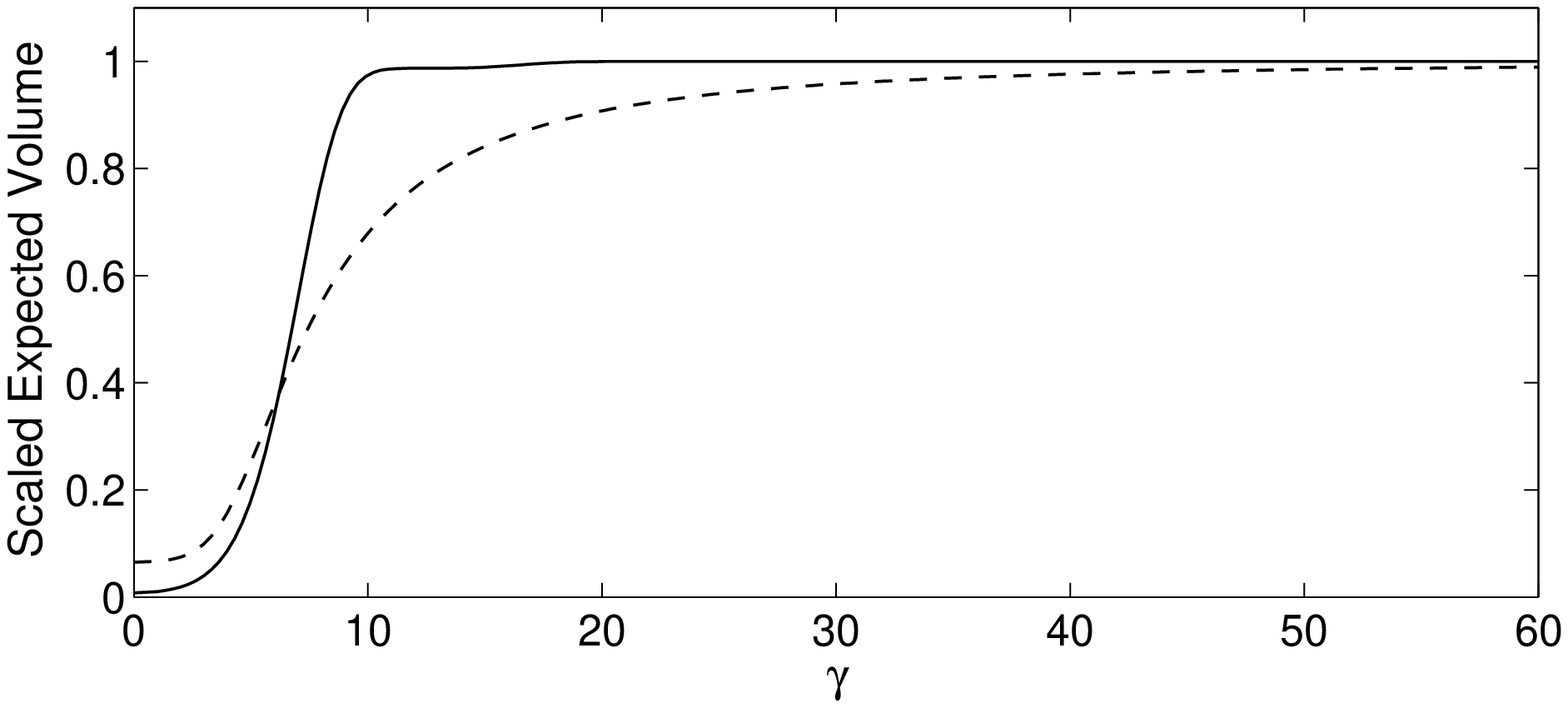} \\
Legend: ------ new RCS \ \ \ - - -  RCS of Casella and Hwang (1983, Section 4)
\end{tabular}
\caption{Plots of the radius function $b$, coverage probability and
scaled expected volume (as functions of $\gamma = || \btheta ||$) for both the new RCS and
the RCS of Casella and Hwang (1983, Section 4),
for $1-\alpha=0.95$ and $p=15$.}
\label{figure 3}
\end{figure}
\clearpage

\medskip

\begin{table}[!htbp]
\centering
\begin{tabular}{*5c}
  \hline
$p$ &  \multicolumn{2}{c}{RCS of Casella \& Hwang (1983, Section 4)} & \multicolumn{2}{c}{new RCS}\\
  \hline
  {}   & minimum CP   & SEV at $\btheta = \boldsymbol{0}$   & minimum CP   & SEV at $\btheta = \boldsymbol{0}$ \\
3	&     0.94594     &     0.88054       &       0.95    &   0.79435\\
5	&     0.94662     &     0.63637       &       0.95    &   0.40805\\
7	&     0.95        &     0.43315       &       0.95    &   0.19833\\
9	&     0.95        &     0.28243       &       0.95    &   0.09245\\
11	&     0.95        &     0.17794       &       0.95    &   0.04127\\
13	&     0.95        &     0.10889       &       0.95    &   0.01831\\
15	&     0.95        &     0.06498       &       0.95    &   0.00804\\
17	&     0.95        &     0.03791       &       0.95    &   0.00357\\
19	&     0.95        &     0.02169       &       0.95    &   0.00272\\	
\hline
\end{tabular}
\caption{Comparison of the new RCS and the RCS of Casella and Hwang (1983, Section 4),
with respect to the minimum coverage probability and the scaled expected volume
at $\btheta = \boldsymbol{0}$,
for $1-\alpha=0.95$ and $p \in \{3,5,7,9,11,13,15,17,19 \}$.}
\end{table}

\smallskip

\noindent {\sl Remark 2.1:}  We have chosen the radius function $b$ to be
a piecewise cubic Hermite interpolating polynomial
in the interval $[0,k]$. Other choices of parametric forms for this function are also possible.
For example, one could choose this function to be a cubic spline in this
interval.

\smallskip

\noindent {\sl Remark 2.2:} Casella and Hwang (1983, Section 3) argue that it is desirable
that the set $S_{\theta}$, described in their Theorem 3.1, is an interval. During the
computation of the new RCS, it was found that at every stage (including the final stage)
this set was an interval.





\bigskip

\noindent {\large{\bf Appendix: Computationally-convenient formula for the scaled expected volume}}

\medskip

The following theorem provide a computationally convenient-formula for the scaled expected volume
of the recentered confidence sphere $J(b)$.

\begin{theorem}

For given function $b$, the scaled expected volume of $J(b)$ is a function of $\gamma = \ntheta$.

 \begin{enumerate}

 \item

 Let $f \big(y;p,\gamma^2 \big)$ denote the noncentral $\chi^2$ pdf with $p$ degrees of freedom and noncentrality
 parameter $\gamma^2$, evaluated at $y$.
 The scaled expected volume of $J(b)$ is
 \begin{equation}
 \label{first_sev}
 \int_0^{\infty} \left( \frac{b \big(\sqrt{y/p} \big)}{d} \right)^p f \big(y;p,\gamma^2 \big) \, dy.
 \end{equation}

 \item

Suppose that $b(x)=d$ for all $x \ge k$, where $k$ is a specified positive number.
The scaled expected volume of $J(b)$ is
 \begin{equation}
 \label{second_sev}
 \int_0^{pk^2} \left( \frac{b \big(\sqrt{y/p} \big)}{d} \right)^p f \big(y;p,\gamma^2 \big) \, dy + 1 - F \big(pk^2; p, \gamma^2 \big),
 \end{equation}
 where $F \big(y; p, \gamma^2 \big)$ denotes the noncentral $\chi^2$ cdf
 with $p$ degrees of freedom and noncentrality
 parameter $\gamma^2$, evaluated at $y$.

 \end{enumerate}

 \end{theorem}

 \begin{proof}

 Note that $Y = ||\bX||^2 = X_1^2 + \dots + X_p^2$ has a noncentral $\chi^2$
 distribution with $p$ degrees of freedom and noncentrality
 parameter $\gamma^2$. It follows from \eqref{sev_initial}
 that the scaled expected volume of $J(b)$ is  \eqref{first_sev}.
 Clearly, \eqref{first_sev} is a function of $\gamma$,
 for given function $b$.

 Suppose that $b(x)=d$ for all $x \ge k$, where $k$ is a specified positive number.
 Obviously,  \eqref{second_sev} follows immediately from \eqref{first_sev}.

 \end{proof}

 Suppose that the radius function $b$ is a piecewise cubic Hermite interpolating polynomial
 in the interval $[0,k]$,
 with knots at $x_1, \dots, x_q$
($0 = x_1 < x_2 < \dots < x_q = k$),
that takes the value $d$ for all $x \ge k$.
This function is very smooth between between successive knots (it is a cubic
between these knots). However, it may not possess a second derivative at each of the knots. For this
reason, we use this theorem to compute the
scaled expected volume of $J(b)$ using the formula
 \begin{equation*}
 \sum_{i=1}^{q-1} \int_{p x_i^2}^{p x_{i+1}^2}
 s(y;b,\gamma) \, dy + 1 - F \big(pk^2; p, \gamma^2 \big),
 \end{equation*}
where each integral is computed separately by numerical quadrature.

\bigskip

\noindent \textbf{References}

\smallskip

\rf BERGER, J. (1980). A robust generalized Bayes estimator and confidence region
for a multivariate normal mean. {\sl Annals of Statistics}, 8, 716--761.



\smallskip

\rf BROWN, L.D. (1966). On the admissibility of invariant estimators of one or more location
parameters. {\sl Annals of Mathematical Statistics}, 37, 1087--1136.

\smallskip

\rf CASELLA, G. and HWANG, J.T. (1983). Empirical Bayes confidence sets for the mean of a multivariate
normal distribution. {\sl Journal of the American Statistical Association}, 78, 688--698.

\smallskip

\rf CASELLA, G. and HWANG, J.T. (1986). Confidence sets and the Stein effect. {\sl Communications in Statistics --
Theory and Methods}, 15, 2043--2063.

\smallskip

\rf CASELLA, G. and HWANG, J.T. (1987). Employing vague prior information in the construction of
confidence sets. {\sl Journal of Multivariate Analysis}, 21, 79--104.

\smallskip

\rf CASELLA, G. and HWANG, J.T. (2012). Shrinkage
confidence procedures. {\sl Statistical Science}, 27, 51--60.

\rf EFRON, B. (2006). Minimum volume confidence regions for a multivariate normal mean vector.
{\sl Journal of the Royal Statistical Society, Series B}, 68, 655--670.

\smallskip

\rf FARCHIONE, D. and KABAILA, P. (2008). Confidence intervals for the normal mean
utilizing prior information. {\sl Statistics \& Probability Letters}, 78, 1094--1100.

\smallskip

\rf FARCHIONE, D. and KABAILA, P. (2012). Confidence intervals in regression centred
on the SCAD estimator. {\sl Statistics \& Probability Letters}, 82, 1953--1960.

\smallskip

\rf FAITH, R.E. (1976).
Minimax Bayes point and set estimators of a multivariate normal mean.
Unpublished PhD thesis, Department of Statistics, University of Michigan.

\smallskip

\rf HWANG, J.T.  and  CASELLA, G.(1982). Minimax
confidence sets for the mean of a multivariate normal distribution.
{\sl Annals of Statistics}, 10, 868--881.

\smallskip

\rf JOSHI, V.M. (1967) Inadmissibility of the usual confidence sets
for the mean of multivariate normal population. {\sl Annals of Mathematical
Statistics}, 38, 1868--1875.

\smallskip

\rf KABAILA, P. and GIRI, K. (2009a). Confidence intervals in regression utilizing
uncertain prior information. {\sl Journal of Statistical Planning and Inference},
139, 3419--3429.

\smallskip

\rf KABAILA, P. and GIRI, K. (2009b). Large-sample confidence intervals in a
two-period crossover, utilizing
uncertain prior information. {\sl Statistics \& Probability Letters},
79, 652--658.

\smallskip

\rf KABAILA, P. and GIRI, K. (2013a). Simultaneous confidence intervals for the population
cell means, for two-by-two factorial data, that utilize
uncertain prior information. To appear in {\sl Communications in Statistics: Theory and Methods}.

\smallskip

\rf KABAILA, P. and GIRI, K. (2013b). Further properties of frequentist confidence intervals
in regression that utilize
uncertain prior information. To appear in {\sl Australian \& New Zealand Journal of Statistics}.

\smallskip

\rf SAMWORTH, R. (2005). Small confidence sets for the mean of a spherically
symmetric distribution. {\sl Journal of the Royal Statistical Society, Series B},
67, 343--361.

\smallskip

\rf SHINOZAKI, N. (1989). Improved confidence sets for the mean of a multivariate
normal distribution. {\sl Annals of the Institute of Mathematical Statistics}, 41,
331--346.

\smallskip

\rf STEIN, C. (1962). Confidence sets for the mean of a multivariate normal
distribution.
{\sl Journal of the Royal Statistical Society, Series B}, 9, 1135--1151.

\smallskip

\rf TSENG, Y.L. and BROWN, L.D. (1997). Good exact confidence sets for a multivariate normal mean.
{\sl Annals of Statistics}, 5, 2228--2258.

\end{document}